\documentclass[11pt]{article}
\usepackage{amsmath, amsthm, amssymb, amsfonts, mathrsfs, mathtools}
\usepackage{graphicx}
\usepackage{makeidx}
\usepackage[left=2cm,right=2cm,top=2cm,bottom=2cm]{geometry}
\usepackage{caption}
\usepackage{subcaption}
\usepackage[section]{placeins}
\usepackage{enumitem}
\usepackage{tikz}
\usepackage{stmaryrd}
\usepackage{authblk}

\usetikzlibrary{decorations.pathreplacing}
\tikzstyle{vertex}=[auto=left,circle,draw=black,fill=white, inner sep=1.5]

\usepackage{hyperref}
\hypersetup{colorlinks=true, citecolor={blue}, linkcolor=blue, filecolor=magenta, urlcolor=cyan}

\newtheorem{theorem}{Theorem}[section]

\newtheorem{lema}[theorem]{Lemma}

\renewcommand{\i}{\mathbf{i}}

\newcommand{\Cay}{\operatorname{Cay}}

\title{A formula for eigenvalues of integral Cayley graphs over abelian groups}

\author[1]{Priya}
\author[2]{Monu Kadyan}
\affil[ ]{\small{\textsuperscript{1}Department of Mathematics, Indian Institute of Science Education and Research Bhopal, India.}}
\affil[ ]{\textsuperscript{2}Yau Mathematical Sciences Center, Tsinghua University, China.}
\affil[ ]{ {\textsuperscript{1}priya22@iiserb.ac.in, \textsuperscript{2}monu4math@gmail.com}}

\date{}

\begin{document}
	\maketitle
	
	\vspace{-0.3in}
	
\begin{center}{\textbf{Abstract}}\end{center} 

Let $Z$ be an abelian group, $ x \in Z$, and $[x] = \{ y : \langle x \rangle = \langle y \rangle \}$. A graph is called integral if all its eigenvalues are integers. It is known that a Cayley graph is integral if and only if its connection set can be express as union of the sets $[x] $. In this paper, we determine an algebraic formula for eigenvalues of the integral Cayley graph when the connection set is $ [x]$. This formula involves an analogue of M$\ddot{\text{o}}$bius function.

%Let $Z=\mathbb{Z}_{n_1} \otimes \mathbb{Z}_{n_2} \otimes \cdots \otimes \mathbb{Z}_{n_k}$ be an abelian group and $ \alpha  \in Z$. Define $\psi_{\alpha}: Z \to \mathbb{C}^*$ such that \( \psi_{\alpha}(y)=\prod\limits_{j=1}^{k} \exp\left(\frac{2\pi \textbf{i}~ \alpha_j ~ y_j }{n_j}\right) \textnormal{ for all $y  \in Z$} \). For $\alpha, x\in Z$, define $[x] = \{ y : \langle x \rangle = \langle y \rangle \}$ and $C_{\alpha}(x) = \sum\limits_{s \in [x]} \psi_{\alpha}(s)$. It is known that eigenvalues of the integral Cayley graph $\text{Cay}(Z, S)$ can be express in the sum of $C_{\alpha}(x)$'s. In this paper, we discuss a formula for the $C_{\alpha}(x)$.  The formula is as follows:
%\begin{align}
%C_{\alpha}(x) = \left\{ \begin{array}{cl}
%		|[x]| & \mbox{if }\langle \bar{\alpha} \rangle = \langle x \rangle  \\
%		(-1)^r ~ |[x]| ~  \prod\limits_{i=1}^{r}   \frac {|M_i|}{| \langle \bar{\alpha} \rangle | - |M_i|}     & \mbox{if } \langle \bar{\alpha} \rangle = \bigcap\limits_{j=1}^{r} M_j  \\
%		0   &\mbox{otherwise} 
%	\end{array}\right. \nonumber
%\end{align} where $\bar{\alpha}$ is a generator of the cyclic group $ \{ y : \psi_{\alpha}(y) =1 \} \cap \langle x \rangle$ and $M_1, M_2, \ldots , M_r$ are some maximal subgroups of $\langle x \rangle $.
%\noindent 

\vspace*{0.3cm}

\section{Introduction}

Eigenvalues of Cayley graphs gained popularity due to their significance in algebraic graph theory and applications in expanders, chemical graph theory, and quantum computing. A large number of results on eigenvalues of Cayley graphs was made in the last few decades, see \cite{liu2022eigenvalues} for a survey. Many algebraic methods, including representation theory and association schemes can be used to study the eigenvalue of Cayley graphs. In this paper, we use an analogue of M$\ddot{\text{o}}$bius inversion formula to construct an algebraic formula for the eigenvalues of integral Cayley graphs.

Throughout this paper, we consider $Z $ to be an abelian group of order $n$. For a positive integer $r$ and $x\in Z$, we write $rx$ to denote the $r$ times sum of $x$ to itself, additive inverse of $x$ by $-x$. We use  $\langle x \rangle$ to denote the cyclic group generated by $x$.

Let $S$ be a close under inverse subset of $Z$. The \textit{Cayley graph} ${\rm Cay}(Z,S)$ is an undirected graph, where $V({\rm Cay}(Z,S))=Z$ and $$E({\rm Cay}(Z,S))=\{ (a,b)\colon a,b\in Z, b-a \in S \}.$$
The \textit{eigenvalues} of the graph are the eigenvalues of its $(0,1)$-adjacency matrix. A graph is called \textit{integral} if all its eigenvalues are integers.

Consider the equivalence relation $\sim$ on the abelian group $Z$ such that $x\sim y$ if and only if $y=mx$ for some integer $m$ such that $m$ is co-prime to the order of $x$. For any $x\in Z$, we use $[x]$ to denote the equivalence class of $x$ with respect to the relation $\sim$. That is, $[x] = \{ y : \langle x \rangle = \langle y \rangle \}$. Define $\mathcal{A}(x)$ to be a collection of the representatives of the equivalence classes of $\sim$ in the cyclic group $\langle x \rangle$. For example, if $Z=\mathbb{Z}_n$ and $x=1$, then $\mathcal{A}(1)=\{ d: d \mbox{ divides } n \}$.

In 1982, Bridge and Meena gave a characterisation of a Cayley graph whose all eigenvalues are integers. This result is as follows.

%\cite{alperin2012integral},
\begin{theorem}\label{Cayint} (\cite{bridges1982rational})
Let $Z$ be an abelian group and $S$ be a close under inverse subset of $Z$. The Cayley graph $\text{Cay}(Z, S)$ is integral if and only if $S=[x_1]\cup [x_2]\cup \ldots \cup [x_r]$ for some $x_1, x_2, \ldots , x_r \in S$.
\end{theorem}

\noindent Consider $Z=\mathbb{Z}_{n_1} \times \mathbb{Z}_{n_2} \times \cdots \times \mathbb{Z}_{n_k}$. For $\alpha \in Z$, define a character $\psi_{\alpha}: Z \to \mathbb{C}^*$ such that
\begin{equation}
\psi_{\alpha}(y)=\prod_{j=1}^{k}\omega_{n_j}^{\alpha_j y_j} \textnormal{ for all $\alpha=( \alpha_1,...,\alpha_k),~y=(y_1, \ldots ,y_k) \in Z$},\nonumber \label{character}
\end{equation}
and $\omega_{n_j}=\exp\left(\frac{2\pi \i}{n_j}\right)$. Note that $\{ \psi_{\alpha}: \alpha \in Z \}$ is the complete set of non-equivalent irreducible characters of the abelian group $Z$. For any $\alpha , x \in Z$, define 
\begin{align}
C_{\alpha}(x):= \sum_{s \in [x]} \psi_{\alpha}(s).\nonumber
\end{align}

\noindent Consider $Z=\mathbb{Z}_n$ and $x=1$. We have $$C_{\alpha}(1)= \sum_{ \substack{ 1\leq s \leq n\\ \gcd(s,n)=1} } \omega_n^{s\alpha} = R_n(\alpha),$$ where $R_n(\alpha)$ is the Ramanujan sum. It was introduced by Ramanujan in 1918. For more information, see \cite{ramanujan1918certain}. It is known that $$R_n(\alpha)= \varphi(n) \frac{\dot{\mu}({\frac{n}{\delta_{\alpha}}})}{\varphi(\frac{n}{\delta_{\alpha}})},$$ where $\delta_{\alpha} = \gcd(n,\alpha)$, $\dot{\mu}$ is classic M$\ddot{\text{o}}$bius function, and $\varphi$ is Euler's phi function. Hence, we have a formula for $C_{\alpha}(1)$. Therefore,  it is natural to ask whether a similar formula can be obtained for $C_{\alpha}(x)$ in which the group $Z$ is non-cyclic. In this paper, we answer it in an affirmative form. In 1979, Babai used characters to determine the eigenvalues of a Cayley graph. This result is as follows.

%foster2016spectra
\begin{theorem}[\cite{babai1979spectra}] \label{EigenvalueExpression}
Let $Z$ be an abelian group and $S$ be a close under inverse subset of $Z$. The eigenvalue of Cayley graph $\Cay(Z, S)$ is $\{ \lambda_\alpha \colon \alpha \in Z \}$, where $$\lambda_{\alpha}=\sum_{s\in S} \psi_{\alpha}(s)  \textnormal{ for each } \alpha \in Z.$$
\end{theorem}

\noindent By Theorem~\ref{Cayint} and Theorem~\ref{EigenvalueExpression}, it is clear that $C_{\alpha}(x)$ is integer and also an eigenvalue of the integral Cayley graph $\Cay(Z, [x])$. Thus, eigenvalues of an integral Cayley graph $\Cay(Z, S)$ can be express as
\begin{align}
\lambda_{\alpha}=\sum_{i=1}^{r} \sum_{s\in [x_i]} \psi_{\alpha}(s) \label{EigIntegCayGra}
\end{align} for some $x_1, x_2, \ldots , x_r \in S$. By Equation~(\ref{EigIntegCayGra}), we get
\begin{align}
\lambda_{\alpha}=\sum_{i=1}^{r}  C_{\alpha}(x_i)~ \textnormal{ for each } \alpha \in Z. \nonumber
\end{align} Hence the sum $C_{\alpha}(x)$ plays important role in eigenvalues of an integral Cayley graph. 

The classic M$\ddot{\text{o}}$bius function is a multiplicative function with the domain of positive integers. It was introduced by August Ferdinand M$\ddot{\text{o}}$bius in 1832. This paper discusses a similar function, taking the domain as a pair of the elements of abelian groups. We call it an analogue of M$\ddot{\text{o}}$bius function. This analogue is the following. For $x,y \in Z$ and $x \in \langle y \rangle$,  define
$$\mu(x,y):= \left\{ \begin{array}{cl}
		1 & \mbox{if }\langle x \rangle = \langle y \rangle  \\
		(-1)^r & \mbox{if }  \langle x \rangle  = \langle q_1q_2\ldots q_r y \rangle   \\
		0   &\mbox{otherwise,} 
	\end{array}\right. $$ where $q_1, q_2, \ldots , q_r$ are some distinct prime divisors of the order of $y$. Note that $$\mu(x,y)=\dot{\mu}\bigg(\frac{\mbox{order of } y}{\mbox{order of } x} \bigg),$$ where $\dot{\mu}$ is classic M$\ddot{\text{o}}$bius function. Similarly, define 
$$\varphi(x,y):= \left\{ \begin{array}{cl}
		1 & \mbox{if }\langle x \rangle  = \langle y \rangle  \\
		(q_1-1)(q_2-1)\ldots (q_r-1) & \mbox{if }  \langle x \rangle  = \langle q_1q_2\ldots q_r y \rangle  \\
		1   &\mbox{otherwise} 
	\end{array}\right. $$  where $q_1, q_2, \ldots , q_r$ are some distinct prime divisors of the order of $y$. The main result of the paper is the following.

\noindent \textbf{Theorem} \textit{(Theorem~\ref{MainResultProof}).  Let $Z$ be an abelian group and $ \alpha,  x  \in Z$. Then 
\begin{align}
C_{\alpha}(x) &= \mu( \bar{\alpha} , x)  ~ \frac{|[x]| }{ \varphi( \bar{\alpha} , x) }, \nonumber
\end{align} where $\bar{\alpha}$ is a generator of the cyclic group $ \{ y : \psi_{\alpha}(y) =1 \}  \cap \langle x \rangle$.}

%We discuss an analogue of M$\ddot{\text{o}}$bius inversion formula in Section 2. Use it as a tool we prove the main result of the paper.

In Section 2, we present an analogue of the M$\ddot{\text{o}}$bius inversion formula. We will use it to prove the main result in Section 3.

%\noindent \textbf{Theorem} \textit{(Theorem~\ref{MainResultProof}). Let $Z$ be an abelian group and $ \alpha,  x  \in Z$. Then 
%\begin{align}
%C_{\alpha}(x) = \left\{ \begin{array}{cl}
%		|[x]| & \mbox{if }\langle \bar{\alpha} \rangle = \langle x \rangle  \\
%		(-1)^r ~ |[x]| ~  \prod\limits_{i=1}^{r}   \frac {|M_i|}{| \langle \bar{\alpha} \rangle | - |M_i|}     & \mbox{if } \langle \bar{\alpha} \rangle = \bigcap\limits_{j=1}^{r} M_j  \\
%		0   &\mbox{otherwise} 
%	\end{array}\right. \nonumber
%\end{align} where $\bar{\alpha}$ is a generator of the cyclic group $ \{ y : \psi_{\alpha}(y) =1 \} \cap \langle x \rangle$ and $M_1, M_2, \ldots , M_r$ are some maximal subgroups of $\langle x \rangle $.}

%%%%%%%%%%%%%%%%%%%%%%%%%%%%%%%%%%%%%%%%%%%%%%%%%%%%%%%%%%%%%%%
%%%%%%%%%%%%%%%%%%%%%%%%%%%%%%%%%%%%%%%%%%%%%%%%%%%%%%%%%%%%%%%

\section{Preliminaries}

\begin{lema}\label{muSumSubset10} Let $Z$ be an abelian group and $x,y, \alpha \in Z$. The following statements are true:
\begin{enumerate}[label=(\roman*)]
\item $$ \sum_{\substack{ x\in \mathcal{A}( y ) \\ \langle \alpha \rangle  \subseteq \langle x \rangle }}  \mu(x,y)= \left\{ \begin{array}{cl}
		1 &  \mbox{if } \langle \alpha \rangle = \langle y \rangle \\
		0   &\mbox{otherwise.} 
	\end{array}\right. $$ 	
	
\item $$ \sum_{\substack{ y\in \mathcal{A}( \alpha ) \\ \langle x \rangle  \subseteq \langle y \rangle }}  \mu(x,y)= \left\{ \begin{array}{cl}
		1 &  \mbox{if } \langle \alpha \rangle = \langle x \rangle \\
		0   &\mbox{otherwise.} 
	\end{array}\right. $$ 	
\end{enumerate}
\end{lema}

\begin{proof}
\begin{enumerate}[label=(\roman*)]
\item If $\langle \alpha \rangle = \langle y \rangle$ then $$ \sum_{\substack{ x\in \mathcal{A}( y  ) \\ \langle \alpha \rangle  \subseteq \langle x \rangle }}  \mu(x,y)=   \mu(y,y) = 1.$$ Assume that $\langle \alpha \rangle \subsetneqq \langle y \rangle$.  We have
$$\sum_{\substack{ x\in \mathcal{A}( y ) \\ \langle \alpha \rangle  \subseteq \langle x \rangle }}  \mu(x,y)= \sum_{\substack{ x\in \mathcal{A}( y ) \\ \langle \alpha \rangle  \subseteq \langle x \rangle \\ \mu(x,y)\neq 0 }}  \mu(x,y).$$ Let $\langle q_1 y \rangle, \langle q_2 y \rangle, \ldots , \langle q_r y \rangle$ are the only maximal subgroups of $\langle y \rangle$ which  contains $\langle \alpha \rangle$, where $q_1,q_2,\ldots , q_r$ are some prime divisors of the order of $y$. For any $0 \leq j \leq r$, ${r \choose j}$ many distinct $x$ exist in $\mathcal{A}( y )$ with $\mu (x,y)=(-1)^j$ ($\langle x \rangle $ can be express as $j$ many maximal subgroups of $\langle y \rangle $). Therefore, we get
$$ \sum_{\substack{ x\in \mathcal{A}( y ) \\ \langle \alpha \rangle  \subseteq \langle x \rangle \\ \mu(x,y)\neq 0 }}  \mu(x,y)  = \sum_{j=0}^{r} {r \choose j} (-1)^j=(1-1)^{r}=0.$$

\item If $\langle \alpha \rangle = \langle x \rangle$ then $$ \sum_{\substack{ y\in \mathcal{A}(\langle \alpha \rangle ) \\ \langle x \rangle  \subseteq \langle y \rangle }}  \mu(x,y)=   \mu(\alpha , \alpha) = 1.$$ Assume that $\langle x \rangle \subsetneqq \langle \alpha \rangle$. 
Then $\langle x \rangle = \langle q_1^{b_1}q_2^{b_2} \ldots q_r^{b_r}  \alpha \rangle$, where $q_1,q_2,\ldots , q_r$ are some distinct prime divisors such that $q_i^{b_i}$ divides the order of $\alpha$ and $b_i \geq 1$ for each $i=1,2,\ldots , r$. Let $\alpha' = q_1^{b_1-1}q_2^{b_2-1} \ldots q_r^{b_r-1}  \alpha$. We get  
\begin{align}
 \sum_{\substack{ y\in \mathcal{A}( \alpha ) \\ \langle x \rangle  \subseteq \langle y \rangle }}  \mu(x,y)  = \sum_{\substack{ y\in \mathcal{A}( \alpha ) \\ \langle x \rangle  \subseteq \langle y \rangle  \\ \mu(x,y)\neq 0 }}  \mu(x,y) =  \sum_{\substack{ y\in \mathcal{A}( \alpha' ) \\ \langle x \rangle  \subseteq \langle y \rangle }}  \mu(x,y). \nonumber 
 \end{align}  Here, the last equality holds from the fact that $\mu(x,y) \neq 0$ if and only if $y\in \mathcal{A}( \alpha' )$.
Therefore, $\langle q_1 \alpha' \rangle, \langle q_2 \alpha' \rangle, \ldots , \langle q_r \alpha' \rangle$ are the only maximal subgroups of $\langle \alpha' \rangle$ which  contains $\langle x \rangle$. Let $\mathcal{S}=\{ \langle q_1 \alpha' \rangle, \langle q_2 \alpha' \rangle, \ldots , \langle q_r \alpha' \rangle\}$. Note that if $y\in \mathcal{A}( \alpha' )$ and $\langle x \rangle  \subseteq \langle y \rangle$ then $\langle y \rangle = \bigcap\limits_{M \in \mathcal{T}} M$ for some $\mathcal{T} \subseteq \mathcal{S}$. It implies that $\mu(x,y) = (-1)^{r-|\mathcal{T}|}$. Therefore, for any $0 \leq j \leq r$, ${r \choose j}$ many distinct $y$ exist in $\mathcal{A}( \alpha')$ satisfying $\mu(x,y)=(-1)^{r-j}$. We get
 \begin{align}
 \sum_{\substack{ y\in \mathcal{A}( \alpha' ) \\ \langle x \rangle  \subseteq \langle y \rangle }}  \mu(x,y)= \sum_{j=0}^{r} {r \choose j} (-1)^{r-j}=(-1+1)^{r}=0. \nonumber
 \end{align} 
\end{enumerate}
\end{proof}

%%%%%%%%%%%%%%%%%%%%%%%%%%%%%%%%%%%%%%%%%%%%%%%%%%%%%%%%%%%%%%%
%%%%%%%%%%%%%%%%%%%%%%%%%%%%%%%%%%%%%%%%%%%%%%%%%%%%%%%%%%%%%%%
%%%%%%%%%%%%%%%%%%%%%%%%%%%%%%%%%%%%%%%%%%%%%%%%%%%%%%%%%%%%%%%

%\section{An analogue of M$\ddot{\text{o}}$bius inversion formula}

The classic M$\ddot{\text{o}}$bius inversion formula is a relation between two arithmetic functions, both functions are defined from the other by taking sums over divisors. It was introduced by August Ferdinand M$\ddot{\text{o}}$bius in 1832. In the cyclic group $\mathbb{Z}_n$, all of the devisor of $n$ can be taken as representatives of the equivalence classes of $\sim$. In general, if $Z$ be an abelian group and $x\in Z$, then $\mathcal{A}(x)=\{ x^d: d \mbox{ divides the order of } x\}$. Keeping this feature in mind, we define an analogue of M$\ddot{\text{o}}$bius inversion formula that aligns two functions of the domain of abelian groups. The analogue is as follows.

\begin{theorem}\label{MobiusInvForm}
Let $Z$ be an abelian group and $f,g: Z \to \mathbb{C}$. Then 
\begin{align}
f(x)= \sum_{y \in \mathcal{A}( x )} g(y) \label{MobiusInvForEq1}
\end{align}
if and only if
\begin{align}
  g(x) = \sum_{y \in \mathcal{A}( x )} f(y)~ \mu(y,x). \label{MobiusInvForEq2}
 \end{align}
\end{theorem}
\begin{proof}
Assume that Equation (\ref{MobiusInvForEq1}) holds. We have
\begin{align}
 \sum_{y \in \mathcal{A}( x )} f(y) ~ \mu(y,x) &=   \sum_{y \in \mathcal{A}( x  )}   \sum_{\alpha \in \mathcal{A}( y )} g(\alpha) ~ \mu(y,x) \nonumber\\
 &=  \sum_{\alpha \in \mathcal{A}( x )}  \sum_{ \substack{ y \in \mathcal{A}( x ) \\ \langle \alpha \rangle  \subseteq \langle y \rangle } }    g(\alpha)  ~ \mu(y,x) \nonumber\\
 &=  \sum_{\alpha \in \mathcal{A}( x )}  g(\alpha) \sum_{ \substack{ y \in \mathcal{A}( x ) \\ \langle \alpha \rangle  \subseteq \langle y \rangle } }      \mu(y,x) \nonumber\\
 &= g(x).\nonumber
\end{align}
Here the first equality follows from Equation (\ref{MobiusInvForEq1}) and last equality follows from Part $(i)$ of Lemma~\ref{muSumSubset10}. Conversely, assume that Equation (\ref{MobiusInvForEq2}) holds. We have
\begin{align}
\sum_{y \in \mathcal{A}( x  )} g(y)  &= \sum_{y \in \mathcal{A}( x  )}    \sum_{\alpha \in \mathcal{A}( y )} f(\alpha) ~ \mu(\alpha, y)  \nonumber\\
&= \sum_{\alpha \in \mathcal{A}( x )}  \sum_{ \substack{ y \in \mathcal{A}( x ) \\ \langle \alpha \rangle  \subseteq \langle y \rangle } }  f(\alpha) ~ \mu(\alpha, y) \nonumber\\
&= \sum_{\alpha \in \mathcal{A}( x )} f(\alpha) \sum_{ \substack{ y \in \mathcal{A}( x ) \\ \langle \alpha \rangle  \subseteq \langle y \rangle } }   \mu(\alpha, y) \nonumber\\
&= f(x).
\end{align}
Here the first equality follows from Equation (\ref{MobiusInvForEq2}) and last equality follows from Part $(ii)$ of Lemma~\ref{muSumSubset10}. 
\end{proof}

\begin{lema}\label{GroupSizeEqEquiLema} Let $Z$ be an abelian group and $ x \in Z$. Then
\begin{align}
|[x]| = \sum_{y \in \mathcal{A}( x )}     |\langle y \rangle | ~ \mu(y,x).   
\end{align} 
\end{lema}

\begin{proof}
We have 
\begin{align}
|\langle x \rangle | = \sum_{y \in \mathcal{A}( x )} |[y]|. \nonumber
\end{align} 
Now the result follows from Theorem~\ref{MobiusInvForm}.
\end{proof}

%%%%%%%%%%%%%%%%%%%%%%%%%%%%%%%%%%%%%%%%%%%%%%%%%%%%%%%%%%%%%%%
%%%%%%%%%%%%%%%%%%%%%%%%%%%%%%%%%%%%%%%%%%%%%%%%%%%%%%%%%%%%%%%
%%%%%%%%%%%%%%%%%%%%%%%%%%%%%%%%%%%%%%%%%%%%%%%%%%%%%%%%%%%%%%%

\section{Proof of Main Theorem}

Let $S$ be a subgroup of the abelian group. Define $$S^{\perp} = \{ x: \psi_s(x)=1 \mbox{ for all } s\in S\}.$$ For example, if $Z=\mathbb{Z}_5 \otimes \mathbb{Z}_5 \otimes \mathbb{Z}_{25}$ then $$\langle (1,0,5) \rangle ^{\perp} = \langle (1 , 1, 4) \rangle \cup \langle (1 , 2, 4) \rangle \cup \langle (1 , 3, 4) \rangle \cup \langle (1 , 4, 4) \rangle  \cup \langle (1 , 0, 4) \rangle \cup \langle (0 ,1, 0) \rangle.$$

\noindent For any $\alpha , x \in Z$, define $$f_{\alpha}(x) :=  \sum\limits_{s\in \langle x \rangle } \psi_{\alpha}(s).$$ We have 
\begin{align}
f_{\alpha}(x) = \left\{ \begin{array}{cl}
		|\langle x \rangle | &  \mbox{if }  \alpha \in \langle x \rangle ^{\perp} \\
		0   &\mbox{otherwise.} 
	\end{array}\right.  \nonumber
\end{align} 
We can also write 
\begin{align}
f_{\alpha}(x) = \sum\limits_{s\in \langle x \rangle } \psi_{\alpha}(s)= \sum_{y \in \mathcal{A}( x ) }   \sum\limits_{s\in [ y ] } \psi_{\alpha}(s)=  \sum_{y \in \mathcal{A}( x ) } C_{\alpha}(y). \nonumber
\end{align}

\noindent By Theorem \ref{MobiusInvForm}, we get
\begin{align}
C_{\alpha}(x) &= \sum_{y \in \mathcal{A}( x ) } f_{\alpha}(y) \mu(y,x) \nonumber\\
&= \sum_{ \substack{ y \in \mathcal{A}( x ) \\ \alpha \in \langle y \rangle ^{\perp} } }  |\langle y \rangle|  \mu(y,x) \nonumber\\ 
&= \sum_{ \substack{ y \in \mathcal{A}( x ) \\ y \in \langle \alpha \rangle ^{\perp} } }  |\langle y \rangle|  \mu(y,x).\label{SubFormulaEq1}
\end{align}

%%%%%%%%%%%%%%%%%%%%%%%%%%%%%%%%%%%%%%%%%%%%%%%%%%%%%%%%%%%%

\noindent Now onwards, we will use the following terminologies:
\begin{itemize}
\item We consider $Z=Z_{n_1} \times Z_{n_2} \times \cdots \times Z_{n_k}$ with $n=n_1 \ldots n_k$, where $Z_{n_i}$ is an abelian group of order $n_i = p_i^{a_i}$ with $a_i \geq 1$ for each $i=1, \ldots , k$. And also, $p_1, p_2, \ldots , p_k$ are distinct primes.
\item We consider  $Z_{n_i}= \mathbb{Z}_{n_{i,1}} \times \cdots \times \mathbb{Z}_{ n_{i,k_i} }$ with $n_i=n_{i,1}\ldots n_{i,k_i}$, where $n_{i,j}=p_i^{a_{i,j}}$ with $a_{i,j} \geq 1$ for each $j=1, \ldots , k_i$. Note that $a_{i,1} + \cdots + a_{i,k_i}=a_i$ for each $i=1, \ldots , k$.
%\item Both $0$ or $p_i^{a_i}$ are consider as the same element of $\mathbb{Z}_{n_i}$.
\item We write the elements $x\in Z $ as elements of the cartesian product $Z_{n_1} \times Z_{n_2} \times \cdots \times Z_{n_k}$, $i.e.$ 
 $$x=(x_1,x_2,\ldots ,x_k),  \mbox{ where } x_i \in Z_{n_i} \mbox{ for all } i \in \{1, \ldots , k\}. $$
 And also, we write the elements $x_i \in Z_{n_i} $ as elements of the cartesian product $\mathbb{Z}_{n_{i,1}} \times \cdots \times \mathbb{Z}_{ n_{i,k_i} }$, $i.e.$ 
  $$x_i=(x_{i,1},x_{i,2},\ldots ,x_{i,k_i}),  \mbox{ where } x_{i,j} \in \mathbb{Z}_{n_{i,j}} \mbox{ for all } j \in \{1, \ldots , k_i\}. $$

\item For any $x,y \in Z$ and $x\in \langle y \rangle$, we say $$\gcd(x_i, n_i) = \gcd(y_i, n_i)$$ if it satisfies 
\begin{align}
 \gcd(x_{i,j}, n_i) = \gcd(y_{i,j}, n_i) \textnormal{ for all } j =1,2, \ldots , k_i. \label{eq7New}
\end{align}
\item If $x_i$ and $y_i$ does not satisfy Equation~(\ref{eq7New}) then we say $\gcd(x_i, n_i) \neq \gcd(y_i, n_i)$.

\item For $x, y \in Z $ and $x \in \langle y \rangle$, define $1_{x,y}'$ and $1_{x,y}''$ such that $i$-th coordinate of $1_{x,y}'$ is given by
\begin{align}
1_{x,y}'(i) = \left\{ \begin{array}{cl}
		(0,0,\ldots , 0) &  \mbox{if } \gcd (  x_i , n_i ) = \gcd ( y_i , n_i) \\
		(1,1,\ldots , 1)   &\mbox{otherwise,} 
	\end{array}\right.  \nonumber
\end{align} 
and $i$-th coordinate of $1_{x,y}''$ is given by
\begin{align}
1_{x,y}''(i) = \left\{ \begin{array}{cl}
		(1,1,\ldots , 1) &  \mbox{if }  \gcd (  x_i , n_i ) = \gcd ( y_i , n_i ) \\
		  (0,0,\ldots , 0)  &\mbox{otherwise.} 
	\end{array}\right.  \nonumber
\end{align}  

\end{itemize}

\begin{lema}\label{MuNonZeroCond} Let $x,y \in Z$, $x\in \langle y \rangle$, $\gcd(x_{i,j}, n_i)=p_i^{b_{i,j}}$, and $\gcd(y_{i,j}, n_i)=p_i^{c_{i,j}}$. Then  $\mu(x,y) \neq 0$ if and only if $b_{i,j} - c_{i,j} =0$  or $1$ for all $i$ and  $j$.
\end{lema}	
\begin{proof} It is clear that $b_{i,j} - c_{i,j} \geq 0$ for all $i$ and $j$. Assume that $\mu(x,y) \neq 0$. If $\langle x \rangle = \langle y \rangle $ then $b_{i,j} - c_{i,j}=0$ for all $i$ and $j$. And so, we are done. Now, assume that $\langle x \rangle \subsetneqq \langle y \rangle $. Assumption implies that $\langle x \rangle = \langle q_1 q_2\ldots q_r y \rangle$, where $q_1,q_2,\ldots q_r$ are the distinct prime divisors of the order of $y$. Therefore $\gcd(z_{i,j}, n_i)=p_i^{b_{i,j}}$ where $z=q_1 q_2\ldots q_r y$. Using the fact that $q_1,q_2,\ldots q_r$ are distinct primes, we get $b_{i,j} - c_{i,j} =0$ or $1$ for all $i$ and $j$. Conversely, assume that  $b_{i,j} - c_{i,j}=0$ or $1$ for all $i$ and $j$. If $y_{i,j} = 0 $, then $b_{i,j} - c_{i,j}=0$.  Without loss of generality, let $r$ be the smallest positive integer such that $b_{i,j} - c_{i,j}=1$ for all $i \leq r$ and some $j$. Therefore, $b_{i,j} - c_{i,j}=0$ for all $ i >r$ and $j$. And so,  $\langle x \rangle = \cap_{i=1}^{r} \langle p_i y \rangle $. Hence $\mu(x,y)=(-1)^r$, which is non-zero.
\end{proof}

\begin{lema}\label{MuZeroCond}
Let $x,y \in Z$ and $x\in \langle y \rangle$. If $\mu(x,y)=0$ then $\mu(tx,y)=0$ for all $t \in \mathbb{N}$.
\end{lema}
\begin{proof}
Assume that $\mu(x,y)=0$ and $t \in \mathbb{N}$. Let $\gcd(x_{i,j}, n_i)=p_i^{b_{i,j}}$, $\gcd(y_{i,j}, n_i)=p_i^{c_{i,j}}$, and $\gcd(tx_{i,j}, n_i)=p_i^{d_{i,j}}$. By Lemma~\ref{MuNonZeroCond}, we get $b_{i,j} - c_{i,j}  \geq 2$ for some $i$ and $j$. Note that $d_{i,j} \geq b_{i,j}$. Therefore, $d_{i,j} - c_{i,j}  \geq 2$ for some $i$ and $j$. Again, Lemma~\ref{MuNonZeroCond} implies that $\mu(tx,y)=0$.
\end{proof}

%%%%%%%%%%%%%%%%%%%%%%%%%%%%%%%%%%%%%%%%%%%%%%%%%%%%%%%%%%%%

%%%%%%%%%%%%%%%%%%%%%%%%%%%%%%%%%%%%%%%%%%%%%%%%%%%%%%%%%%%%

%%%%%%%%%%%%%%%%%%%%%%%%%%%%%%%%%%%%%%%%%%%%%%%%%%%%%%%%%%%

%\begin{lema}
%  Let $x , y  \in Z$, and $x\in \langle y \rangle$. The following statements %are true:
%  \begin{enumerate} 
%      \item $|\langle x \rangle | = |\langle x1_{x,y}' \rangle | ~ |\langle %x1_{x,y}'' \rangle |$.
%      \item $\mu(x,y)= \mu(x1_{x,y}',y1_{x,y}') ~ %\mu(x1_{x,y}'',y1_{x,y}'')$.
%      \item $\mu(x1_{x,y}' ,y1_{x,y}' )=\mu(x,y)$.
%  \end{enumerate}
%\begin{proof} Let $x'=x1'_{x,y}$ and $y'' = y 1''_{x,y} $.
%    \begin{enumerate}
%       \item Easy to prove.
%        \item 
%        \item
%    \end{enumerate}
%\end{proof}

\begin{lema}\label{NewResultHelpFinal}
Let $x , y  \in Z$, and $x\in \langle y \rangle$. If $\mu(x,y) \neq 0$ then $|\langle x1'_{x,y} \rangle | ~ |[y 1''_{x,y} ]| = \frac{|[y]|}{\varphi(x,y)} $.
\end{lema}
\begin{proof} Let $x'=x1'_{x,y}$ and $y'' = y 1''_{x,y} $.  We have the following two cases:\\	
\textbf{Case 1:} Assume that $\gcd (  x_i , n_i ) = \gcd ( y_i , n_i)$ for all $i =1,2,\ldots , k$. Therefore, $\langle x \rangle = \langle y \rangle $.  Then $ x' =(0,0,\ldots, 0)$ and $y''  = y$. The result follows from the fact that $\varphi(x,y) =1$. \\
\noindent \textbf{Case 2:}  Without loss of generality, assume that $\gcd (  x_i , n_i ) \neq \gcd ( y_i , n_i)$ for all $i \leq r$ and $\gcd (  x_i , n_i ) = \gcd ( y_i , n_i)$ for all $i > r$.  Using the fact that $\mu(x,y) \neq 0$, we get  $\langle x \rangle = \langle p_1p_2\ldots p_r y \rangle $. Let $| \langle y \rangle | = \prod_{i=1}^{k} p_i^{b_i} $ where $b_i \geq 0$ for all $i=1,2,\ldots , k$. Note that $b_i \geq 1$ for all $i=1,2,\ldots , r$. We have $| \langle x \rangle | = \prod_{i=1}^{r} p_i^{b_i-1}  \prod_{i=r+1}^{k} p_i^{b_i}$, $| \langle x' \rangle | = \prod_{i=1}^{r} p_i^{b_i-1} $, and $| \langle y'' \rangle | = \prod_{i=r+1}^{k} p_i^{b_i}$. It implies that $$
|[y]|=  \prod_{i=1 ~ \& ~ b_i \neq 0 }^{k} (p_i-1) p_i^{b_i-1} \mbox{ and } | [ y'' ] | = \prod_{\substack{ i=r+1 ~ \& ~ b_i \neq 0  }}^{k} (p_i-1) p_i^{b_i-1}.$$ The result follows from the fact that $\varphi(x,y) = \prod_{i=1}^{r} (p_i-1)$.
\end{proof}

%%%%%%%%%%%%%%%%%%%%%%%%%%%%%%%%%%%%%%%%%%%%%%%%%%%%%%%%%%%%

\begin{theorem}\label{MainResultProof} Let $Z$ be an abelian group and $ \alpha,  x  \in Z$. Then 
\begin{align}
C_{\alpha}(x) &= \mu( \bar{\alpha} , x)  ~ \frac{|[x]| }{ \varphi( \bar{\alpha} , x) }, \nonumber
\end{align} where $\bar{\alpha}$ is a generator of the cyclic group $ \langle \alpha \rangle ^{\perp} \cap \langle x \rangle$.
\end{theorem}

\begin{proof} Assume that $\langle \bar{\alpha} \rangle = \langle \alpha \rangle ^{\perp} \cap \langle x \rangle$. The Equation~(\ref{SubFormulaEq1}) implies 
\begin{align}
C_{\alpha}(x) = \sum_{ \substack{ y \in \mathcal{A}( x ) \\ y \in \langle \bar{\alpha} \rangle  } }  |\langle y \rangle|  ~ \mu(y,x). \label{SubFormulaEq1134}
\end{align}
If $y \in \langle \bar{\alpha} \rangle$ then $y = y_1 \bar{\alpha} $ for some $y_1 \in \mathbb{N} $. Equation~(\ref{SubFormulaEq1134}) implies

\begin{align}
C_{\alpha}(x) &= \sum_{ \substack{ y \in \mathcal{A}( x ) \\ y = y_1 \bar{\alpha}  } }  |\langle y  \rangle| ~ \mu( y ,x). \label{SubFormulaEq112} 
\end{align}

If $\mu( \bar{\alpha}, x) = 0$ then by Lemma~\ref{MuZeroCond} $\mu( t  \bar{\alpha} ,x) = 0$ for all  $t \in \mathbb{N}$. Equation~(\ref{SubFormulaEq112}) implies that $C_{\alpha}(x) =0$. So, assume that $\mu( \bar{\alpha} ,x) \neq 0$. Define $\bar{\alpha}' = \bar{\alpha} 1_{\bar{\alpha},x}'$, $\bar{\alpha}'' = \bar{\alpha} 1_{\bar{\alpha},x}''$, $x' = x 1_{\bar{\alpha},x}'$, and $x'' = x 1_{\bar{\alpha},x}''$. Note that $\bar{\alpha} = \bar{\alpha}' + \bar{\alpha}''$ and $x=x' + x''$. We get 

\begin{align}
C_{\alpha}(x) &= \sum_{ \substack{ y \in \mathcal{A}( x' + x'' ) \\ y = y_1 (\bar{\alpha}'+ \bar{\alpha}'')  } }  |\langle  y  \rangle| ~ \mu( y  ,x' + x'') \nonumber \\ 
&= \sum_{ \substack{ u \in \mathcal{A}( x' ) ~ \& ~ v \in \mathcal{A}( x'' ) \\ u = y_1 \bar{\alpha}' ~ \& ~ v = y_1  \bar{\alpha}''   } }    |\langle  u+v  \rangle| ~ \mu( u+v  ,x' + x'') \nonumber \\
&= \sum_{ \substack{ u \in \mathcal{A}( x' ) ~ \& ~ v \in \mathcal{A}( x'' ) \\ u = y_1 \bar{\alpha}' ~ \& ~ v = y_1  \bar{\alpha}''   } }   |\langle  u \rangle| ~  |\langle  v  \rangle| ~ \mu(  u  ,x') ~ \mu( v  ,x''). \label{SubFormulaEq113} 
\end{align}

Assume that $u \in \mathcal{A}( x' )$, $ u = y_1 \bar{\alpha}' $, and $\mu( u ,x') \neq 0$. Let $i\in \{ 1, 2, \ldots , k \}$. We have $\gcd (\bar{\alpha}'_{i,j} , n_i) > \gcd( x'_{i,j} , n_i )$ for some $j$ whenever $\bar{\alpha}'_{i} \neq (0,0, \ldots , 0)$. Lemma~\ref{MuNonZeroCond} implies that $\gcd( y_1, n_i )=1 $ whenever $\bar{\alpha}'_{i} \neq (0,0,\ldots , 0)$. It implies that $y_1$ is co-prime to the order of $\bar{\alpha}'$, and so $ \langle u  \rangle =  \langle  \bar{\alpha}' \rangle $. We get $ |\langle u  \rangle| =  |\langle  \bar{\alpha}' \rangle| $ and $ \mu( u, x') =\mu( \bar{\alpha}' ,x') $. Using the fact that $\mu( \bar{\alpha}' ,x') = \mu( \bar{\alpha} ,x) $, Equation~(\ref{SubFormulaEq113}) implies that 
\begin{align}
C_{\alpha}(x) &=  \mu( \bar{\alpha}  ,x) ~  |\langle  \bar{\alpha}' \rangle|    \sum_{ \substack{ v \in \mathcal{A}( x'' ) \\ v = y_1 \bar{\alpha}''  } }   |\langle  v  \rangle| ~  \mu( v  ,x''). \label{SubFormulaEq114} 
\end{align}
\noindent We have $\langle \bar{\alpha}'' \rangle = \langle x'' \rangle$. It implies
\begin{align}
  \sum_{ \substack{ v \in \mathcal{A}( x'' ) \\ v = y_1 \bar{\alpha}''  } }   |\langle v  \rangle|  ~ \mu( v  ,x'')  =    \sum_{ v \in \mathcal{A}( x'' )  }   |\langle v  \rangle|  ~ \mu( v  ,x'')     = |[x'']| .   \label{SubFormulaEq115} 
\end{align}
Here the equality holds from Lemma~\ref{GroupSizeEqEquiLema}. Now, substitute  the Equation~(\ref{SubFormulaEq115}) in Equation~(\ref{SubFormulaEq114}), we get 
\begin{align}
C_{\alpha}(x) &= \mu( \bar{\alpha} , x) ~ |\langle \bar{\alpha}' \rangle | ~  |[x'']|. \nonumber
\end{align} 
The result follows from Lemma~\ref{NewResultHelpFinal}.
\end{proof}

%%%%%%%%%%%%%%%%%%%%%%%%%%%%%%%%%%%%%%%%%%%%%%%%%%%%%%%%%%%%%%%%%%%%%%%%%%%%%%%%%%%%%%
%
%\bibliographystyle{plain}
%\bibliography{SeprateBibFile}

\end{document}